\documentclass[final, a4paper, 10pt]{article}

\usepackage[scale=0.768]{geometry}
\usepackage[english]{babel}

\usepackage[mathscr]{euscript}
\usepackage{color}
\usepackage{fancyhdr}
\usepackage{bm}
\usepackage{hyperref}
\usepackage{cite}
\usepackage{showkeys}
\usepackage{subfigure}
\usepackage{float}

\usepackage{amsfonts}
\usepackage{amsmath}
\usepackage{amssymb}
\usepackage{amscd}
\usepackage{amsthm}

\usepackage{accents}
\usepackage{graphicx}
\usepackage{array}

\usepackage{algorithm}
\usepackage{algorithmicx}
\usepackage{algpseudocode}

\newtheorem{theorem}{Theorem}
\newtheorem{definition}[theorem]{Definition}
\newtheorem{lemma}[theorem]{Lemma}

\newtheorem{assumption}[theorem]{Assumption}

\newcommand*{\N}{\ensuremath{\mathbb{N}}}
\newcommand*{\Z}{\ensuremath{\mathbb{Z}}}
\newcommand*{\R}{\ensuremath{\mathbb{R}}}
\newcommand*{\C}{\ensuremath{\mathbb{C}}}
\renewcommand{\i}{\mathrm{i}}
\renewcommand{\phi}{\varphi}
\renewcommand{\rho}{{\varrho}}
\renewcommand{\epsilon}{{\varepsilon}}

\renewcommand{\d}[1]{\,\mathrm{d}#1 \,}

\newcommand{\J}{\mathcal{J}} % Bloch transform

\newcommand{\0}{{0}} %  {0}
\newcommand{\B}{{\mathcal{B}}}

\renewcommand{\j}{{\bm j}}

\renewcommand{\B}{{\mathfrak{B}}}

\newcommand{\p}{{per}}
\renewcommand{\L}{\mathcal{L}} 
\renewcommand{\Re}{\mathrm{Re}\,}
\renewcommand{\Im}{\mathrm{Im}\,}

\newlength{\dhatheight}

\usepackage{color}

\definecolor{xl}{rgb}{0.8,0.2,0.3}

\usepackage{amsmath}
\usepackage{amsfonts}
\usepackage{amssymb}
\begin{document}
	
	\sloppy
	
\title{Exponential convergence of perfectly matched layers for scattering problems with periodic surfaces}
\author{
Ruming Zhang\thanks{Institute of Applied and Numerical mathematics, Karlsruhe Institute of Technology, Karlsruhe, Germany
; \texttt{ruming.zhang@kit.edu}. }}
\date{}
\maketitle
	
\begin{abstract}
 The main task in this paper is to prove that the perfectly matched layers (PML) method converges exponentially with respect to the PML parameter for scattering problems with periodic surfaces.  In \cite{Chand2009}, a linear convergence is proved for the PML method for scattering problems with rough surfaces. At the end of that paper, three important questions were asked, and the third question is if exponential convergence holds locally. In our paper, we answer this question for a special case, when the rough surface is actually periodic. The result can also be easily extended to locally perturbed periodic surfaces or layers. Due to technical reasons, we have to exclude all the wavenumbers which are half integers. The main idea of the proof is to apply the Floquet-Bloch transform to rewrite the problem as an equivalent family of quasi-periodic problems, and then study the analytic extension of the quasi-periodic problems with respect to the Floquet-Bloch parameters. Then the Cauchy integral formula is applied to avoid linear convergent points. Finally the exponential convergence is proved from the inverse Floquet-Bloch transform. Numerical results are also presented at the end of this paper. \\

 \noindent
 {\bf Keywords: PML method, scattering problems, periodic surfaces, exponential convergence, Cauchy integral theorem}

\end{abstract}

\section{Introduction}

The PML method has been widely applied to the simulation of wave propagations in unbounded domains since it was invented in \cite{Beren1994}. The main idea for the PML method is to add an artificial absorbing layer outside the scatterers, where there is almost no reflection. The problem is then approximated by the truncation with a proper boundary condition. To guarantee that the PML method works, it is of essential importance to study  the convergence of the solution with respect to the PML parameters.

This paper studies the convergence of the PML method for acoustic scattering problems with periodic surfaces in two dimensional spaces. This paper is motivated by the open questions at the end of \cite{Chand2009}. Although a linear convergence has been proved for rough surfaces in that paper, exponential convergence was shown for an extreme case when the surface is flat. The authors conjectured that the exponential convergence held locally also for nonflat surfaces. In this paper, we try to answer this question for periodic surfaces, using techniques introduced in \cite{Chen2003}. First, we introduce the setting of this problem as well as some important notations and spaces.

Suppose $\Gamma$ is a surface defined by a $2\pi$-periodic Lipschitz continuous function $\zeta$, and $\Omega$ is the unbounded periodic domain above $\Gamma$:
\[
\Gamma:=\big\{(x_1,\zeta(x_1)):\,x_1\in\R\big\};\quad \Omega:=\big\{(x_1,x_2):\,x_2>\zeta(x_1):\,x_1\in\R\big\}.
\]
%Without loss of generality we assume $\zeta>0$ on $\R$. 
%Given an incident field $u^i$, it is scattered by the periodic surface $\Gamma$ which is assumed to be nonpenetrable, and produces the scattered field $u^s$. 
For simplicity, we only consider the problem described by the following model:
\begin{equation}
 \label{eq:waveguide}
 \Delta u+k^2 u=f\text{ in }\Omega;\quad u=0\text{ on }\Gamma,
\end{equation}
where $f\in L^2(\Omega)$ is a compactly supported source term.  

Let $H$ be a   number  such that  $$H>\sup_{x_1\in\R}\zeta(x_1)\quad\text{ and } \quad H>\sup_{x\in {\rm supp}(f)}x_2.$$ Let $\Gamma_H:=\R\times\{H\}$ be a straight line lying above $\Gamma$ and let the periodic strip between $\Gamma$ and $\Gamma_H$ be denoted by $\Omega_H$. Then ${\rm supp}(f)\subset \Omega_H$. Thus $u$ satisfies the homogeneous Helmholtz equation  when $x_2>H$.

To guarantee that the solution $u$ propagates upwards, we also require that $u$ satisfies the following radiation condition (see \cite{Chand2005}):
\[
 u(x_1,x_2)=\int_\R \widehat{u}(\xi,H)e^{\i \xi x_1+\i\sqrt{k^2-\xi^2}(x_2-H)}\d \xi,\quad x_2\geq H,
\]
where $\widehat{u}(\xi,H)$ is the Fourier transform of $u(x_1,H)$ and $\sqrt{k^2-\xi^2}$ has non-negative real and imaginary parts. This radiation condition defines the following DtN map on $\Gamma_H$:
\[
 (T^+\phi)(x_1)=\i\int_\R\sqrt{k^2-\xi^2}\,\widehat{\phi}(\xi) e^{\i \xi x_1}\d\xi,\quad\text{ where }\phi(x_1)=\int_\R\widehat{\phi}(\xi) e^{\i \xi x_1}\d\xi.
\]
From \cite{Chand2005}, $T^+$ is a bounded operator from $H^{1/2}(\Gamma_H)$ to $H^{-1/2}(\Gamma_H)$. Thus $u$  satisfies the following boundary condition:
\begin{equation}
 \label{eq:bc}
 \frac{\partial u}{\partial x_2}=T^+u\quad\text{ in }\quad H^{-1/2}(\Gamma_H).
\end{equation}
Now the problem is formulated in the periodic domain $\Omega_H$ with finite height by \eqref{eq:waveguide}-\eqref{eq:bc}. The weak formulation is straight forward, i.e., to find $u\in \widetilde{H}^1(\Omega_H)$ such that
\begin{equation}
 \label{eq:var}
 \int_{\Omega_H}\left[\nabla u\cdot\nabla\overline{\phi}-k^2  u \overline{\phi}\right]\d x-\int_{\Gamma_H}\left[T^+ u\right]\overline{\phi}\d s=-\int_{\Omega}f\overline{\phi}d x
\end{equation}
holds for any compactly supported $\phi\in \widetilde{H}^1(\Omega_H)$, where
\[
 \widetilde{H}^1(\Omega_H):=\left\{\psi\in H^1(\Omega_H):\,\psi\big|_{\Gamma}=0\right\}.
\]
From \cite{Chand2005} it is known that the problem \eqref{eq:var} is uniquely solvable in $\widetilde{H}^1(\Omega_H)$. For unique solvability in the weighted Sobolev space $\widetilde{H}^1_r(\Omega_H)$ ($|r|<1$)  we refer to \cite{Chand2010}.

We apply the Floquet-Bloch transform to \eqref{eq:var}, and the problem is written as a family of quasi-periodic problems, and the original solution is then written as the inverse Floquet-Bloch transform of quasi-periodic problems, which is an integral on an interval with respect to the Floquet-Bloch parameters. The quasi-periodic problems depend piecewise smoothly on the Floquet-Bloch parameters, with only one or two square root singularities (later called ``cutoff values''). At the cutoff values, only linear convergence are proved for the PML method. For parameters away from those points, exponential convergence is proved. For details we refer to \cite{Chen2003,Chand2009,Kirsc2021}. To the best of the author's knowledge, the exponential convergence for rough surfaces (including periodic surfaces) is only proved for complex valued wavenumbers, for details we refer to \cite{Li2011}.

To deal with these points, we  extend the quasi-periodic problems analytically with respect to the Floquet-Bloch parameters. With the help of the Cauchy integral formula, the inverse Floquet-Bloch transform equals to an integral on a modified contour. We design the  contour carefully such that it has a positive distance to the cutoff values. From technical reasons, we have to assume that the wavenumber is not a half integer. Then we prove the uniform exponential convergence for parameters lying on the contour, which finally results in the exponential convergence for the PML method.

At the end of this paper, several numerical examples are presented to show that the PML method actually converges exponentially. From these results, exponential convergence is shown for wavenumbers including half integers. The convergence rate is also far better than expected. This leads to a possible further topic, which is to extend  the method to half integers and also to prove the sharper estimates.

The rest of this paper is organized as follows. In the second section, we apply the Floquet-Bloch transform to the problem. In Section 3, the transformed problems are extended analytically and then the inverse Floquet-Bloch transform is modified from the Cauchy integral formula. The exponential convergence is proven then in Section 4. In Section 5, numerical examples are presented. Some further discussions and comments are shown in the last section.

\section{The Floquet-Bloch Transform}

In this section, we apply the Floquet-Bloch transform to the problem \eqref{eq:waveguide}-\eqref{eq:bc}, or equivalently, \eqref{eq:var}. For simplicity, we first define the domains restricted to one periodicity cell. 
Let $\Gamma_j$, $\Omega_j$ and $\Omega_H^j$ be the restriction of $\Gamma$, $\Omega$ and $\Omega_H$ to one periodicity cell $\big[2\pi j-\pi,2\pi j+\pi\big]\times\R$. Without loss of generality, assume that ${\rm supp}(f)\subset\Omega_H^0$.

%It is well known that this problem is not always uniquely solvable for positive

%We first discuss the case when $n=1$. In this case, the problem is uniquely solvable for given compactly supported $L^2$ function $f$ and the solution $u\in \widetilde{H}^1_r(\Omega_H)$ where $r\in(0,1)$ and this is the weighted Sobolev space, the tilde means $u=0$ on $\Gamma$ (\cite{Chand2010}).

Recall the definition of the Floquet-Bloch transform $\J$ for compactly supported smooth function $\phi$:
\[
( \J\phi)(\alpha,x):=\sum_{j\in\Z}\phi\left(x+\left(\begin{matrix}
                                                     2\pi j\\0
                                                    \end{matrix}
\right)\right)e^{-\i\alpha(x_1+2\pi j)},\quad x\in\Omega_H^0;\,\alpha\in [-1/2,1/2].
\]
Here $\alpha$ is called the Floquet-Bloch parameter throughout this paper. 
It has been proved  (see Theorem 4.2, \cite{Lechl2015e}) that $\J$ is an isomorphism between $H^s(\Omega_H)$ and $L^2\left([-1/2,1/2];H^s_\p(\Omega_H^0)\right)$. Note that the space $H^s_\p(\Omega_H^0)\subset H^s(\Omega_H^0)$ contains all functions that are $2\pi$-periodic in $x_1$-direction, and the space $L^2\left([-1/2,1/2];H^s_\p(\Omega_H^0)\right)$ is equipped with the norm:
\[
 \|\psi\|^2_{L^2\left([-1/2,1/2];H^s_\p(\Omega_H^0)\right)}=\int_{-1/2}^{1/2}\left\|\psi(\alpha,\cdot)\right\|^2_{H^s_\p(\Omega_H^0)}\d\alpha.
\]
The subspace $L^2\left([-1/2,1/2];\widetilde{H}^s_\p(\Omega_H^0)\right)$ contains all the functions $\phi\in L^2\left([-1/2,1/2];H^s_\p(\Omega_H^0)\right)$ such that $\phi(\alpha,\cdot)\big|_{\Gamma_0}=0$. 
Moreover, the inverse Floquet-Bloch transform coincides with the adjoint operator of $\J$, i.e.,
\[
 \left(\J^{-1}\psi\right)(x)=\int_{-1/2}^{1/2}\psi(\alpha,x)e^{\i\alpha x_1}\d\alpha,\quad x\in\Omega_H.
\]

Given any compactly supported $f\in L^2(\Omega_0)$, the problem \eqref{eq:var} has a unique solution $u\in \widetilde{H}^1(\Omega_H)$ (see \cite{Chand2005}). 
Let $w:=\J u$ then $w\in L^2\left([-1/2,1/2];\widetilde{H}^1_\p(\Omega_H^0)\right)$. In \cite{Lechl2016}, we also proved that $w$ depends continuously on $\alpha\in[-1/2,1/2]$. For all $\alpha\in[-1/2,1/2]$, $w(\alpha,\cdot)$ is $2\pi$-periodic  with respect to $x_1$. Moreover, $w(\alpha,\cdot)$ is the strong solution of:
\begin{eqnarray}
 \label{eq:per1}
 &&\Delta w(\alpha,\cdot)+2\i \alpha\frac{\partial w(\alpha,\cdot)}{\partial x_1}+(k^2 n-\alpha^2) w(\alpha,\cdot)=e^{-\i\alpha x_1}f\text{ in }\Omega_H^0;\\
 \label{eq:per2}
 && w(\alpha,\cdot)=0\text{ on }\Gamma_0;\\
 \label{eq:per3}
 && \frac{\partial w(\alpha,\cdot)}{\partial x_2}={T}^+_\alpha w(\alpha,\cdot)\text{ on }\Gamma_H^0.
\end{eqnarray}
Note that here $T^+_\alpha$ is the $\alpha$-dependent periodic DtN map given by:
\[
 (T^+_\alpha\phi)(x_1)=\i\sum_{j\in\Z}\sqrt{k^2-(\alpha+j)^2} \widehat{\phi}_j e^{\i j x_1}\quad\text{ where } \quad\phi(x_1)=\sum_{j\in\Z}\widehat{\phi}_j e^{\i j x_1}
\]
where
\[
 \sqrt{k^2-(\alpha+j)^2}=\begin{cases}
        \sqrt{k^2-(\alpha+j)^2},\quad\text{ if }|\alpha+j|\leq k;\\
        \i\sqrt{(\alpha+j)^2-k^2},\quad \text{ if }|\alpha+j|>k.
                         \end{cases}
\]
The operator $T^+_\alpha$ is bounded from $H^{1/2}_\p(\Gamma_H^0)$ to $H^{-1/2}_\p(\Gamma_H^0)$.

It is already known that the problem \eqref{eq:per1}-\eqref{eq:per3} is uniquely solvable in $\widetilde{H}^1_\p(\Omega_H^0)$ for given $f\in L^2(\Omega_H^0)$. We refer to \cite{Kirsc1993,Bruck2003} for details. With these solutions, we get the original solution from the inverse Floquet-Bloch transform, i.e.,
\begin{equation}\label{eq:inverse}
 u(x)=\int_{-1/2}^{1/2} e^{\i\alpha x_1 }w(\alpha,x)\d\alpha,\quad x\in\Omega_H.
\end{equation}

\section{Analytic extension}

In this section, we recall the analytic extension introduced in \cite{Arens2021}. First we give the weak formulation for the $\alpha$-dependent periodic problem, i.e., to find $\widetilde{H}^1_\p\left(\Omega_H^0\right)$ such that
\begin{equation}
 \label{eq:var_per}
 \begin{aligned}
 \int_{\Omega_0}\left[\nabla w(\alpha,\cdot)\cdot\nabla\overline{\phi}-2\i\alpha\frac{\partial w(\alpha,\cdot)}{\partial x_1}\overline{\phi}-(k^2 n-\alpha^2)w(\alpha,\cdot)\overline{\phi}\right]&\d x\\
 -2\pi\i\sum_{j\in\Z}\sqrt{k^2-(\alpha+j)^2}\widehat{w}(\alpha,j)\overline{\widehat{\phi}(j)}&=-\int_{\Omega_0}e^{-\i\alpha x_1}f(x)\overline{\phi(x)}\d x,
 \end{aligned}
\end{equation}
where $\widehat{w}(\alpha,j)$ and $\widehat{\phi}(j)$ are the $j$-th Fourier coefficients of $w(\alpha,\cdot)\big|_{\Gamma_H^0}$ and $\phi\big|_{\Gamma_H^0}$, respectively.

Define the following operators by the Riesz representation theorem:
\begin{eqnarray*}
 && \left<A_1 \psi,\phi\right>= \int_{\Omega_0}\left[\nabla \psi\cdot\nabla\overline{\phi}-k^2 n \psi\overline{\phi}\right]\d x;\\
 &&\left<A_2 \psi,\phi\right>=-2\i\int_{\Omega_0}\frac{\partial\psi}{\partial x_1}\overline{\phi}\d x;\\
 &&\left<A_3 \psi,\phi\right>=\int_{\Omega_0}\psi\overline{\phi}\d x;\\
 && \left<B_j \psi,\phi\right>=-2\pi\i\widehat{\psi}(j)\overline{\widehat{\phi}(j)}.
\end{eqnarray*}
Note that here $\left<\cdot,\cdot\right>$ is the inner product of the space $\widetilde{H}^1_\p(\Omega_H^0)$. 
Then all the operators are bounded in $\widetilde{H}^1_\p(\Omega_H^0)$ and independent of $\alpha$. 
There is also a family of elements $F(\alpha,\cdot)\in \widetilde{H}^1_\p(\Omega_H^0)$ such that
\[
 \left<F(\alpha,\cdot),\phi\right>=-\int_{\Omega_0}e^{-\i\alpha x_1}f(x)\overline{\phi(x)}\d x.
\]
Since $e^{-\i\alpha x_1}f(x)$ depends analytically on $\alpha$, also  $F$ depends analytically on $\alpha$. Then \eqref{eq:var_per} is written as the following $\alpha$-dependent equations:
\begin{equation}
 \label{eq:oper_per}
 \left(A_1+\alpha A_2+\alpha^2 A_3+\sum_{j\in\Z}\sqrt{k^2-(\alpha+j)^2}B_j\right)w(\alpha,\cdot)=F(\alpha,\cdot).
\end{equation}
For simplicity set 
\[
 D(\alpha):=A_1+\alpha A_2+\alpha^2 A_3+\sum_{j\in\Z}\sqrt{k^2-(\alpha+j)^2}B_j.
\]

We know that $D(\alpha)$ is invertible for all $\alpha\in[-1/2,1/2]$ and the solution $w(\alpha,\cdot)$ has square root singularities at the $\alpha\in[-1/2,1/2]$ when $|\alpha+j|=k$ for some $j\in\Z$ (for details see \cite{Kirsc1993}). Since $A_1,A_2,A_3,B_j$ are independent of $\alpha$, the singularities only come from the coefficients in front of $B_j$. Since the singular points are particularly important, we give the following definition.

\begin{definition}
 \label{def} 
 Any point $\alpha\in[-1/2,1/2]$ such that $|\alpha+j|=k$ for some $j\in\Z$ is called a ``cutoff value``.
\end{definition}

First note that if $k$ is a half integer, for one cutoff value $\alpha\in[-1/2,1/2]$, there are two integers $j_1\neq j_2$ such that $|\alpha+j_1|=|\alpha+j_2|=k$. This case is more complicated and will not be treated in this paper. Thus we  make the following assumption. 

\begin{assumption}\label{asp1}
 Assume that $k\neq\frac{n}{2}$ for all positive integer $n$. 
\end{assumption}

With Assumption \ref{asp1}, $k>0$ can be written as $\kappa+\j$ ($\j\in\N$), where $\kappa\in(-1/2,1/2)\setminus\{0\}$ is called the ``rounding error'' of $k$. Note that the decomposition of the positive number $k$ is unique. From this decomposition, there are two cutoff values, i.e., $-\kappa$ and $\kappa$.

Consider the analytic extension of the solution with respect to $\alpha$ to a small neighbourhood of $[-1/2,1/2]\subset\C$. First we begin with the coefficients of $B_j$. Define: 
\[
 G^+(\alpha,j)=\sqrt{k+\alpha+j},\quad G^-(\alpha,j)=\sqrt{k-\alpha-j}.
\]
Note that the relationship $\sqrt{xy}=\sqrt{x}\sqrt{y}$ does not hold for arbitrary $x$ and $y$. We here use $\sqrt{k^2-(\alpha+j)^2}=G^+(\alpha,j)G^-(\alpha,j)$ since at least one of $k+\alpha+j$ and $k-\alpha-j$ is non-negative. Note that in the following analytic extension, we should also be very careful such that the values of $G^\pm(\alpha,j)$ are not changed for $\alpha\in[-1/2,1/2]$, to guarantee the relationship always holds.

\begin{definition}
 \label{def:sqrt}
 In this paper, the square root ``$\sqrt{\,\,\,\,\,\,}$'' is defined in the branch cutting along the negative imaginary axis. 
\end{definition}

%Now we study the zeros of $G^\pm(\alpha,j)$. From direct calculation:
%\begin{itemize}
% \item When $k=\kappa+\j$, $G^+(-\kappa,-\j)=G^-(\kappa,\j)=0$;
% \item When $k=-\kappa+\j$, $G^+(\kappa,-\j)=G^-(-\kappa,\j)=0$.
%\end{itemize}

We find all the zeros of $G^\pm(\alpha,j)$ for $(\alpha,j)\in\left[(-1/2,1/2)\setminus\{0\}\right]\times\Z$:
\[
 G^+(-\kappa,-\j)=G^-(\kappa,\j)=0.
\]
Now we focus on the analytic extension of $G^+(\alpha,j)G^-(\alpha,j)$ near the zeros $(-\kappa,-\j)$ and $(\kappa,\j)$. Note that
  \[
 G^+(\alpha,-\j)=\sqrt{\kappa+\alpha},\quad G^-(\alpha,-\j)=\sqrt{\kappa+2\j-\alpha}.
\]
and
  \[
 G^+(\alpha,\j)=\sqrt{\kappa+2\j+\alpha},\quad G^-(\alpha,\j)=\sqrt{\kappa-\alpha}.
\]
Note that when $\j=0$,
\[
 G^+(\alpha,-\j)=G^+(\alpha,\j)=\sqrt{\kappa+\alpha},\quad  G^-(\alpha,-\j)= G^-(\alpha,\j)=\sqrt{\kappa-\alpha}.
\]

The discussion is carried out for the following different situations. Define the rays $Z_\pm\subset\C$ by $Z_-:=-\kappa+\i\R_{\leq 0}$ and $Z_+:=\kappa+\i\R_{\geq 0}$. Let $\delta\in(0,|\kappa|)$.

\begin{itemize}
 \item Let $\alpha$ be in a neighourhood of $-\kappa$. 
 \begin{itemize}
  \item From Definition \ref{def:sqrt}, $G^+(\alpha,-\j)$ is analytic  in $[(-\kappa-\delta,-\kappa+\delta)+\i\R]\setminus Z_-$  and $G^-(\alpha,-\j)$ is analytic in $(-\kappa-\delta,-\kappa+\delta)+\i\R$.  %$G^+(\alpha,-\j)$ is extended to an analytic function for $\alpha\in (-\kappa-\delta,-\kappa+\delta)\times\R$ with respect to $\alpha$ in the branch cutting along $\{-\kappa\}\times(-\infty,0)$. 
  Therefore, $G^+(\alpha,-\j)G^-(\alpha,-\j)$ is   analytic in  $[(-\kappa-\delta,-\kappa+\delta)+\i\R]\setminus Z_-$.
 % \item When $j=\j$. If $\j\neq 0$,  both functions are analytic when $\alpha\in (-\kappa-\delta,-\kappa+\delta)\times\R$. When $\j=0$, the case is exactly the same as the case  $j=-\j$, thus the function $G^+(\alpha,-\j)G^-(\alpha,-\j)$ is  extended to an analytic function in a branch cutting along $\{-\kappa\}\times(-\infty,0)$ in the domain $(-\kappa-\delta,-\kappa+\delta)\times\R$.
 \item When $\j\neq 0$ both $G^+(\alpha,\j)$ and $G^-(\alpha,\j)$ are analytic in   $(-\kappa-\delta,-\kappa+\delta)+\i\R$. The case $\j=0$ is treated as in the previous item. 
 \end{itemize}
\item Let $\alpha$ be in a neighourhood of $\kappa$.  
 \begin{itemize}
  \item  From Definition \ref{def:sqrt}, $G^+(\alpha,\j)$  is analytic in $(\kappa-\delta,\kappa+\delta)+\i\R$ and $G^-(\alpha,\j)$ is analytic in $\left[(\kappa-\delta,\kappa+\delta)+\i\R\right]\setminus Z_+$. %$G^-(\alpha,\j)$ is extended to an analytic function for $\alpha\in (\kappa-\delta,\kappa+\delta)\times\R$ with respect to $\alpha$ in the branch cutting along $\{\kappa\}\times(0,+\infty)$. 
  Therefore $G^+(\alpha,\j)G^-(\alpha,\j)$ is  analytic in $\left[(\kappa-\delta,\kappa+\delta)+\i\R\right]\setminus Z_+$.
  \item When $\j\neq 0$ both $G^+(\alpha,-\j)$ and $G^-(\alpha,-\j)$ are analytic in  $(\kappa-\delta,\kappa+\delta)+\i\R$. The case $\j=0$ is treated as in the previous item.  
 \end{itemize}
\end{itemize}

From above arguments, when $k$ satisfies Assumption \ref{asp1} the operator $D(\alpha)$ is extended analytically to $\left[(-1/2-\epsilon,1/2+\epsilon)+\i\R\right]\setminus \left(Z_-\cup Z_+\right)$. Note that a sufficiently small $\epsilon>0$ can be chosen since $\pm1/2\neq\kappa$. For a visualization of the branch cuts we refer to (a), Figure \ref{fig:branch_cuts}. %In the similar way we conclude that when $k=-\kappa+\j$, $D(\alpha)$ is extended analytically to $\alpha\in (-1/2-\epsilon,1/2+\epsilon)+\i\R$ with branch cutting along $\{-\kappa\}+\i\R_{\geq 0}$ and $\{\kappa\}+\i\R_{\leq 0}$. We refer to  (b), Figure \ref{fig:branch_cuts} for a visulization.

\begin{figure}[tttttt!!!b]
\centering
\begin{tabular}{cc}
\includegraphics[width=0.4\textwidth]{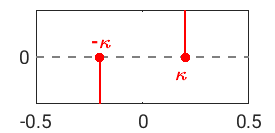} & 
\includegraphics[width=0.4\textwidth]{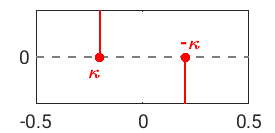} \\
(a) & (b) 
\end{tabular}
\caption{Branch cuts for different settings: (a) $\kappa>0$; (b) $\kappa<0$.}
\label{fig:branch_cuts}
\end{figure}

Before the discussion we introduce some notations. Denote the open disk with center $z_0$ and radius $\delta$  by $\mathfrak{B}(z_0,\delta)$. Moreover, the upper and lower half disks are defined by:
\[
 \mathfrak{B}_+(z_0,\delta):=\left\{z\in \mathfrak{B}(z_0,\delta):\,\Im(z)> 0\right\},\quad  \B_-(z_0,\delta):=\left\{z\in \B(z_0,\delta):\,\Im(z)< 0\right\}.
\]

To consider the analytic extension of $w(\alpha,\cdot)=D^{-1}(\alpha)F(\alpha,\cdot)$ with respect to $\alpha$, we need to separate the operator $D(\alpha)$ by an analytic part and a singular part. First we consider the extension near the point $-\kappa$. Define
\[
 D_+(\alpha):=A_1+\alpha A_2+\alpha^2 A_3+\sum_{j\neq-\j}\sqrt{k^2-(\alpha+j)^2}B_j,
\]
then
\[
  D(\alpha)= D_+(\alpha)+\sqrt{\kappa+\alpha}\,B_+(\alpha)
\]
where 
\[
B_+(\alpha)=G^-(\alpha,-\j)B_{-\j}=\sqrt{\kappa+2\j-\alpha}\,B_{-\j}.
\]
From above formulas, both $D_+$ and $B_+$ depend analytically on $\alpha\in\B(-\kappa,\delta)$ when $\delta>0$ is sufficiently small. Since $D_+$ is a small perturbation of the invertible operator $D(\alpha)$, it is also invertible and depends analytically on $\alpha\in \B(-\kappa,\delta)$ for a small $\delta>0$. From Neumann series, 
\[
 D^{-1}(\alpha)=D_+^{-1}(\alpha)\left[\sum_{n=0}^\infty (-\sqrt{\kappa+\alpha})^n (B_+(\alpha) D_+^{-1}(\alpha))^n\right], \text{ when }0\leq \delta<<1.
\]
Define
\begin{eqnarray*}
 && D_+^1(\alpha)=\sum_{n=0}^\infty (\kappa+\alpha)^n\,D_+^{-1}(\alpha)\left(B_+D_+^{-1}(\alpha)\right)^{2n};\\
 && D_+^2(\alpha)=-\sum_{n=0}^\infty (\kappa+\alpha)^n\,D_+^{-1}(\alpha)\left(B_+D_+^{-1}(\alpha)\right)^{2n+1},
\end{eqnarray*}
then
\[
 D^{-1}(\alpha)=D_+^1(\alpha)+\sqrt{\kappa+\alpha}\,D_+^2(\alpha).
\]
Here both $D_+^1$ and $D_+^2$ depend analytically on $\alpha\in\B(-\kappa,\delta)$ for a small $\delta>0$. Then  the solution has the following decomposition:
\[
 w(\alpha,\cdot)=D^{-1}(\alpha)F(\alpha,\cdot)=w^1_+(\alpha,\cdot)+\sqrt{\kappa+\alpha}\,w^2_+(\alpha,\cdot),
\]
where $w^1_+(\alpha,\cdot)=D^1_+(\alpha)F(\alpha,\cdot)$ and $w^2_+(\alpha,\cdot)=D^2_+(\alpha)F(\alpha,\cdot)$ both depend analytically on $\alpha\in\B(-\kappa,\delta)$ for a small $\delta>0$. Thus $w(\alpha,\cdot)$ depends analytically on $\alpha\in\B(-\kappa,\delta)\setminus Z_-$.

Similarly, %we also study the analytic extension of $w$ in the neighbourhood of $\kappa$. We define similar operators:
%\[ D_-(\alpha):=A_1+\alpha A_2+\alpha^2 A_3+\sum_{j\neq \j}\sqrt{k^2-(\alpha+j)^2}B_j,\]
%and
%\[B_-(\alpha)=G^+(\alpha,\j)B_{\j}=\sqrt{\kappa+2\j+\alpha} B_\j.%\]
%Then 
%\[ D(\alpha)=D_-(\alpha)+\sqrt{\kappa-\alpha}\,B_-(\alpha).\]
%Thus 
in $\B(\kappa,\delta)$ for a small $\delta>0$, $w$ has the decomposition:
\[
 w(\alpha,\cdot)=D^{-1}(\alpha)F(\alpha,\cdot)=w^1_-(\alpha,\cdot)+\sqrt{\kappa-\alpha}\,w^2_-(\alpha,\cdot),
\]
where $w^1_-(\alpha,\cdot)$ and $w^2_-(\alpha,\cdot)$ both depend analytically on $\alpha$. We conclude the results in the following theorem. 

\begin{theorem}
 \label{th:ana_ext}
 Let $k$ satisfy Assumption \ref{asp1} and $k=\kappa+\j$ for some $j\in\N$ and $\kappa\in(-1/2,1/2)\setminus\{0\}$. For fixed $\alpha\in[-1/2,1/2]$,  $w(\alpha,\cdot)\in \widetilde{H}^1_\p(\Omega_H^0)$ is the unique weak solution of \eqref{eq:var_per}. Then $w(\alpha,\cdot)$ is extended analytically to $\B(-\kappa,\delta)\setminus Z_-$   and $\B(\kappa,\delta)\setminus Z_+$. Note that here $0<\delta<|\kappa|\leq k$ is sufficiently small.
 
\end{theorem}

%Now the solution $w(\alpha,\cdot)$ has been extended analytically to small disks cutting along certain half lines. %For $\alpha\in[-1/2,1/2]$ outside the disks, since $D(\alpha)$ depends analytically on $\alpha$ and invertible, $D^{-1}(\alpha)$ exists and extended analytically in a small neighourhood of $\alpha$. This leads to the following corollary.

%\begin{corollary}
% \label{cr:ana_ext}
% Let $k$ satisfies Assumption \ref{asp1} and $w(\alpha,\cdot)\in \widetilde{H}^1_\p(\Omega_H^0)$ is the unique weak solution of \eqref{eq:var_per}. Then there is are sufficiently small $\epsilon,\delta>0$ such that $w(\alpha,\cdot)$ is extended analytically as follows:
% \begin{itemize}
%  \item When $k=\kappa+\j$, then $w(\alpha,\cdot)$ is extended analytically to $(-1/2-\epsilon,1/2+\epsilon)\times(-\delta,\delta)$ cutting along $\{-\kappa\}\times(-\infty,0)$ and $\{\kappa\}\times(0,+\infty)$.
%  \item  When $k=-\kappa+\j$, then $w(\alpha,\cdot)$ is extended analytically to $(-1/2-\epsilon,1/2+\epsilon)\times(-\delta,\delta)$ cutting along $\{-\kappa\}\times(0,+\infty)$ and $\{\kappa\}\times(-\infty,0)$.
% \end{itemize}

%\end{corollary}

In the next step, we will modify the integral in the inverse Floquet-Bloch transform \eqref{eq:inverse} near the cutoff values $-\kappa$ and $\kappa$ with the results in Theorem \ref{th:ana_ext}. First consider the case that $k=\kappa+\j$ where $\kappa\in(0,1/2)$ and $\j\in\N$. For simplicity, 
the discussion begins with a scalar valued function.

\begin{lemma}
 \label{lm:contour} 
 Let the square roots be defined in Definition \ref{def:sqrt}.
 \begin{itemize}
  \item Suppose $g(\alpha)$ is an analytic function defined in a small neighourhood of the half disk $\overline{\B_+(-\kappa,\delta)}$. Define the half circle
\[
 \mathfrak{C}_+:=\Big\{|\alpha+\kappa|=\delta:\,\Im(\alpha)\geq 0\Big\}
\]
with a  clockwise direction. Then  the following equation holds:
 \[
  \int_{\mathfrak{C}_+}\sqrt{\kappa+\alpha}\,g(\alpha)\d\alpha=\int_{-\delta-\kappa}^{\delta-\kappa}\sqrt{\kappa+\alpha}\,g(\alpha)\d\alpha.
 \]
 \item   Suppose $g$ is analytic in a small neighourhood of the half disk $\overline{\B_-(\kappa,\delta)}$. 
 Let \[
 \mathfrak{C}_-:=\left\{|\alpha-\kappa|=\delta:\,\Im(\alpha)\leq 0\right\}
\]
be the half circle with a counter clockwise direction. Then  the following equation holds:
 \[
  \int_{\mathfrak{C}_-}\sqrt{\kappa-\alpha}\,g(\alpha)\d\alpha=\int_{-\delta+\kappa}^{\delta+\kappa}\sqrt{\kappa-\alpha}\,g(\alpha)\d\alpha.
 \]
 \end{itemize}

\end{lemma}

\begin{proof}
 We prove the first item. 
 Let $0<\epsilon<<\delta$ be sufficiently small, and let 
  \[
  {\mathfrak{C}_+^\epsilon}:=\Big\{|\alpha+\kappa|=\epsilon:\,\Im(\alpha)\geq 0\Big\}
 \]
 with a  clockwise direction. 
 Since $\sqrt{\kappa+\alpha}$ is analytic in the domain encircled by $(-\delta-\kappa,-\epsilon-\kappa)$, ${\mathfrak{C}_+^\epsilon}$, $(\epsilon-\kappa,\delta-\kappa)$ and $\mathfrak{C}_+$, from Cauchy integral formula,
 \[
  \int_{\mathfrak{C}_+}\sqrt{\kappa+\alpha}\,g(\alpha)\d\alpha=\left(\int_{-\delta-\kappa}^{-\epsilon-\kappa}+\int_{{\mathfrak{C}_+^\epsilon}}+\int_{\epsilon-\kappa}^{\delta-\kappa}\right)\sqrt{\kappa+\alpha}\,g(\alpha)\d\alpha.
 \]
 Since $\sqrt{\kappa+\alpha}\,g(\alpha)$ depends continuously on $\alpha$ in the half disk and equals to $0$ at $-\kappa$, 
 \[
  \lim_{\epsilon\rightarrow 0^+}\left(\int_{-\delta-\kappa}^{-\epsilon-\kappa}+\int_{\mathfrak{C}_\epsilon^+}+\int_{\epsilon-\kappa}^{\delta-\kappa}\right)\sqrt{\kappa+\alpha}\,g(\alpha)\d\alpha=\int_{-\delta-\kappa}^{\delta-\kappa}\sqrt{\kappa+\alpha}\,g(\alpha)\d\alpha.
 \]
Thus the equation holds.

The proof of the second item is similar thus is omitted.

\end{proof}

The results in Lemma \ref{lm:contour} are easily extended to Banach spaces. For $w(\alpha,\cdot)$ with analytic extension described in Theorem \ref{th:ana_ext}, the following equations are obvious results from Lemma \ref{lm:contour}:
\begin{eqnarray}
 \label{eq:cont_banach1}
 && \int_{\mathfrak{C}_+}e^{\i\alpha x_1}w(\alpha,x)\d\alpha=\int_{-\kappa-\delta}^{-\kappa+\delta}e^{\i\alpha x_1}w(\alpha,x)\d\alpha;\\
 \label{eq:cont_banach2}
 && \int_{\mathfrak{C}_-}e^{\i\alpha x_1}w(\alpha,x)\d\alpha=\int_{\kappa-\delta}^{\kappa+\delta}e^{\i\alpha x_1}w(\alpha,x)\d\alpha.
\end{eqnarray}

At the end of this section, we modify the integral contour in \eqref{eq:inverse} and the results are concluded in the following theorem.

\begin{theorem}
 \label{th:modif_cont}

 $k$ satisfies Assumption \ref{asp1} and $\kappa\in(-1/2,1/2)\setminus\{0\}$ is the rounding error of $k$. Then $k=\kappa+\j$ for some $\j\in\N$. Let $0<\delta<|\kappa|$ be a sufficiently small parameter given in Lemma  \ref{lm:contour}. Define 
  \[
 \mathfrak{C}=\Big([-1/2,1/2]\setminus\big[(-\kappa-\delta,-\kappa+\delta)\cup(\kappa-\delta,\kappa+\delta)\big]\Big)\cup \mathfrak{C}_+\cup \mathfrak{C}_-,
\]
where $\mathfrak{C}_+$ and $\mathfrak{C}_-$ are defined in Lemma \ref{lm:contour}.
 Then the integer contour in \eqref{eq:inverse} is replaced by $\mathfrak{C}$:
\begin{equation}
 \label{eq:new_inverse}
  u\left(x\right)=\int_{\mathfrak{C}} e^{\i\alpha x_1}w(\alpha,x)\d\alpha,\quad x\in\Omega_H.
\end{equation}
 
\end{theorem}

\section{Perfectly matched layers}

In this section we follow the method introduced in \cite{Chen2003} for $\alpha$-dependent periodic problem \eqref{eq:per1}-\eqref{eq:per3}.   Although the arguments were made in \cite{Chen2003} for real-valued $\alpha$, everything is extended to complex valued cases without major differences. We only need to be careful about the new square roots in Definition \ref{def:sqrt}.

We add a PML layer above $\Gamma_H$ with thickness $\lambda>0$. To describe the PML layer, we need the complex valued function $s(x_2)$ defined by:
\[
 s(x_2)=1+\rho \widehat{s}(x_2)
\]
where $\rho>0$ is a parameter, $\widehat{s}(x_2)$ is a sufficiently smooth function which vanishes when $x_2\leq H$. For example, the function can be defined by a polynomial:
\[
 \widehat{s}(x_2)=\chi\left(\frac{x_2-H}{\lambda}\right)^m,\quad x_2\in[H,H+\lambda],
\]
where $\chi$ is a fixed complex number with positive real and imaginary parts and $m$ is a positive integer. For simplicity, let $|\chi|=1$. Define the PML parameter 
\[
 \sigma:=\int_H^{H+\lambda}s(x_2)\d x_2=\lambda\left(1+\frac{\rho\chi}{m+1}\right).
\]
Thus $\sigma=|\sigma|e^{\i\tau}$ where  $\tau\in(0,\pi/2)$ and $|\sigma|\approx (m+1)^{-1}\lambda\rho$ when $\rho>>1$. %Moreover, from the representation, both $\sigma_1$ and $\sigma_2$ are positive. Furthermore, with the assumption that $\Re(\tau)\geq \Im(\tau)>0$, it is easily concluded that $\sigma_1>\sigma_2>0$.

For any fixed $\alpha$, the differential operator with the PML layer with parameter $\sigma$ is defined as follows:
\[
 \L_\sigma(\alpha): =s(x_2)\left(\frac{\partial^2}{\partial x_1^2}+2\i\alpha\frac{\partial}{\partial x_1}\right)+\frac{\partial}{\partial x_2}\left(\frac{1}{s(x_2)}\frac{\partial}{\partial x_2}\right)+(k^2-\alpha^2) s(x_2).
\]
Then the new problem with PML layer is described by the following equation:
\begin{equation}
 \label{eq:pml_or}
 \L_\sigma(\alpha) w^{PML}_\sigma(\alpha,\cdot)=f\text{ in }\Omega_{H+\lambda}^0;\quad w^{PML}_\sigma(\alpha,\cdot)=0\text{ on }\Gamma_0\cup\Gamma_{H+\lambda}^0.
\end{equation}
From \cite{Chen2003}, a solution of \eqref{eq:pml_or} satisfies the boundary condition
\begin{equation}
 \label{eq:pml_bc}
 \frac{\partial w^{PML}_\sigma(\alpha,\cdot)}{\partial x_2}=T^{PML}_{\alpha,\sigma} w^{PML}_\sigma(\alpha,\cdot)\text{ on }\Gamma_H^0,
\end{equation}
where $T^{PML}_{\alpha,\sigma}$ is the $(\alpha,\sigma)$-dependent DtN map defined by:
\[
 (T^{PML}_{\alpha,\sigma}\phi)(x_1)=\i\sum_{j\in\Z}\beta_j\coth\left(-\i\beta_j\,\sigma\right)\widehat{\phi}(j)e^{\i j x_1},\quad \phi(x_1)=\sum_{j\in\Z}\widehat{\phi}(j)e^{\i j x_1}.
\]
From similar arguments in \cite{Chand2009}, it is bounded from $H^{1/2}_\p(\Gamma_H^0)$ to $H^{-1/2}_\p(\Gamma_H^0)$. With the DtN map, the problem \eqref{eq:pml_or}-\eqref{eq:pml_bc} is formulated as the following variational problem in $\Omega_H^0$. That is, to find $w^{PML}_\sigma(\alpha,\cdot)\in \widetilde{H}^1_\p(\Omega_H^0)$ such that
\begin{equation}
 \label{eq:var_per_pml}
 \begin{aligned}
 \int_{\Omega_0}\left[\nabla w^{PML}_\sigma(\alpha,\cdot)\cdot\nabla\overline{\phi}-2\i\alpha\frac{\partial w^{PML}_\sigma(\alpha,\cdot)}{\partial x_1}\overline{\phi}-(k^2 n-\alpha^2)w^{PML}_\sigma(\alpha,\cdot)\overline{\phi}\right]&\d x\\
 -2\pi\i\sum_{j\in\Z}\sqrt{k^2-(\alpha+j)^2}\coth\left(-\i\sqrt{k^2-(\alpha+j)^2}\,\sigma\right)\widehat{w}^{PML}_\sigma(\alpha,j)\overline{\widehat{\phi}(j)}&=-\int_{\Omega_0}e^{-\i\alpha x_1}f\overline{\phi}\d x
 \end{aligned}
\end{equation}
holds for any test function $\phi\in \widetilde{H}^1_\p(\Omega_H^0)$. Compare this problem with \eqref{eq:var_per}, the only difference is the DtN map. Similar to previous arguments, we first define the operator depending on $\sigma$:
\[
 D^{PML}_\sigma(\alpha):=A_1+\alpha A_2+\alpha^2 A_3+\sum_{j\in\Z}\sqrt{k^2-(\alpha+j)^2}\coth\left(-\i\sqrt{k^2-(\alpha+j)^2}\,\sigma\right)B_j,
\]
thus \eqref{eq:var_per_pml} is equivalent to solve the problem
\[
 D^{PML}_\sigma(\alpha)w^{PML}_\sigma(\alpha,\cdot)=F(\alpha,\cdot).
\]
To study the convergence of the PML method, it is equivalent to study the convergence of $D^{PML}_\sigma(\alpha)$ to $D(\alpha)$ with respect to $|\sigma|$, where the key is the convergence of $\coth\left(-\i\sqrt{k^2-(\alpha+j)^2}\,\sigma\right)$ to $1$. In \cite{Chen2003}, it has already been proved that for fixed $\alpha\in[-1/2,1/2]\setminus\{-\kappa,\kappa\}$ the convergence is exponential. However, from \cite{Kirsc2021}, only linear convergence was proved  at the cutoff values. In this paper, we avoid the cutoff values by using the modified inverse Floquet-Bloch transform defined in \eqref{eq:new_inverse}. Thus we only need to prove the exponential convergence of $w^{PML}_\sigma(\alpha,\cdot)$ to $w(\alpha,\cdot)$ for all $\alpha\in \mathfrak{C}$, where $\mathfrak{C}$ is defined in Theorem \ref{th:modif_cont}.

To this end, define the function:
\[
 h(z):=\exp\left(-2\i\sqrt{k^2-z^2}\sigma\right).
\]
Then
\[
 \coth\left(-\i\sqrt{k^2-(\alpha+j)^2}\,\sigma\right)-1=\frac{2}{h(\alpha+j)-1},\quad \alpha\in \mathfrak{C}.
\]
 Let $\mathfrak{C}$ be extended as:
\[
\mathfrak{C}_{ext}:= \cup_{j\in\Z}\left(\mathfrak{C}+\left( j,0 \right)^\top\right)\subset\C,
\]
then $\{\alpha+j:\,\alpha\in \mathfrak{C},\,j\in\Z\}=\mathfrak{C}_{ext}$. Then we estimate $\frac{2}{h(z)-1}$ for any $z\in \mathfrak{C}_{ext}$. The extended curve $\mathfrak{C}_{ext}$ is plotted in Figure \ref{fig:esti_h}.  Here we want to draw the reader's attention to the shape of $\mathfrak{C}_{ext}$. %It always takes the upper half circle near $-k$ and lower circle near $k$, for any $k>0$ such that Assumption \ref{asp1} is satisfied. 
For simplicity, we make the following assumption for the constant $\tau$ (recall that it is the angle of the PML parameter $\sigma$).

\begin{assumption}
 \label{asp2}
 The angle $\tau$ is assumed to line in the interval $\left(\frac{\pi}{8},\frac{\pi-\arctan 2}{2}\right)$.
\end{assumption}

Note that the results in the following lemmas and theorems are always true for any $\tau\in(0,\pi/2)$.  We keep Assumption \ref{asp2} just want to have a simplified process.

\begin{figure}[tttttt!!!b]
\centering
\includegraphics[width=0.8\textwidth]{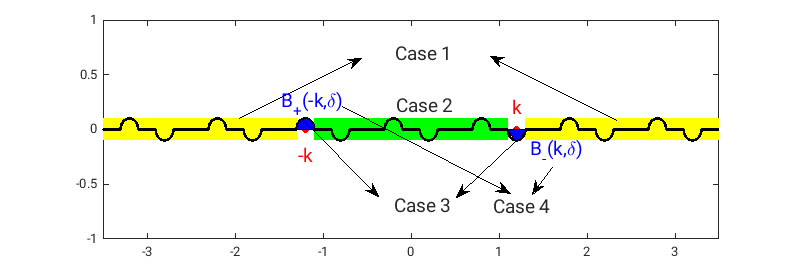} 
\caption{Domain in Lemma \ref{lm:esti_h}. The black curve is $\mathfrak{C}_{ext}$.}
\label{fig:esti_h}
\end{figure}

\begin{lemma}
 \label{lm:esti_h}Suppose $0<\delta<|\kappa|$ defines the curve $\mathfrak{C}$ in Theorem \ref{th:modif_cont}.\\
1)  There is a $\gamma>0$ such that $|h(z)|\geq \exp\left(\gamma \sqrt{\delta}|\sigma|\sqrt{|\Re(z)|+k}\right)$ holds uniformly for $z\in \mathfrak{C}_{ext}$.\\
2) Suppose $\gamma>0$ is the same as in 1), then
\[
 \left|\frac{2\sqrt{k^2-z^2}}{h(z)-1}\right|\leq\gamma^{-1}|\sigma|^{-1}
\]
holds uniformly for $z\in \cup_{j\in\Z}\,\overline{\B_+(-k+j,\delta)}$ and  $ \cup_{j\in\Z}\,\overline{\B_-(k+j,\delta)}$. 
%Here $\delta$ is the same constant as in Theorem \ref{th:modif_cont}. For a visulization of the domain we refer to Figure \ref{fig:esti_h}.
\end{lemma}

\begin{proof}We prove this lemma with four different cases.%For a visulization of the cases we refer to Figure \ref{fig:esti_h}.

\noindent
{\em Case 1.} Let $z\in\C$ such that $|\Re z|\geq k+\delta$ and $|\Im z|\leq\delta$ (yellow domain in Figure \ref{fig:esti_h}). 
\\ Let $z=a+\i b$ where $|a|\geq k+\delta$ and $|b|\leq\delta$. Then
 \[
  \sqrt{k^2-z^2}=\sqrt{k^2+b^2-a^2-2\i ab}.
 \]
 %Since $|b|\leq\delta<k$,  the real part
 %\[
 % k^2+b^2-a^2\leq (k+|a|)(k-|a|)+\delta^2\leq -\delta(k+|a|)+\delta^2< -\frac{\delta (k+|a|)}{4}<0
 %\]
Let $\sqrt{k^2-z^2}=r_z e^{\i\theta}$. From the element computation with the re-definition of the square root,
 \[
  r_z=\left|\sqrt{k^2-z^2}\right|\geq \frac{\sqrt{(k+|a|)\delta}}{2}\quad\text{ and }\quad\theta\in\left(\frac{\pi-\arctan 2}{2},\frac{\pi+\arctan 2}{2}\right)\subset\left(\frac{\pi}{4},\frac{3\pi}{4}\right).
 \]
 This implies that
 \[
  \Re\left(-2\i r_z |\sigma|e^{\i(\theta+\tau)}\right)=2 r_z|\sigma|\sin(\theta+\tau).
 \]
From Assumption \ref{asp2}, 
 \[
  \theta+\tau\in\left(\frac{\pi-\arctan 2}{2}+\frac{\pi}{8},\frac{\pi+\arctan 2}{2}+\frac{\pi-\arctan 2}{2}\right)\subset (0,\pi).
 \]
There is a $\gamma_1>0$ such that $\sin(\theta+\tau)\geq \gamma_1$. This implies that
\[
 \Re\left(-2\i r_z |\sigma|e^{\i(\theta+\tau)}\right)\geq \sqrt{(k+|a|)\delta}|\sigma|\gamma_1.
\]
 Thus
 \[
 | h(z)|=\left|e^{-2\i r_z |\sigma|e^{\i(\theta+\tau)}}\right|\geq \exp\left( \sqrt{(k+|a|)\delta}|\sigma|\gamma_1\right).
 \]

 \noindent
{\em Case 2.} Let $z\in\C$ such that $|\Re z|\leq k-\delta$ and $|\Im z|\leq\delta$ (green domain in Figure \ref{fig:esti_h}).\\ Let $z=a+\i b$ where $|a|<k-\delta$ and $|b|<\delta$. %Then $$k^2+b^2-a^2\geq (k+|a|)(k-|a|)\geq \delta(k+|a|)>2\delta|a|\geq 0.$$ Thus 
Similarly $\sqrt{k^2-z^2}=r_z e^{\i\theta}$ where
 \[
  r_z=\left|\sqrt{k^2-z^2}\right|\geq \sqrt{\delta(|a|+k)}\quad\text{ and }\quad\theta\in\left(-\frac{\pi}{8},\frac{\pi}{8}\right).
 \]
% Thus
 %\[
 % |h(z)|=\left|e^{-2\i r_z|\sigma| e^{\i(\theta+\tau)}}\right|\geq \exp\left(2 \sqrt{\delta(|a|+k)}|\sigma|%\sin\left(\theta+\tau\right)\right).
% \]
From Assumption \ref{asp2} again we can also find a $\gamma_2>0$ such that $\sin(\theta+\tau)\geq \gamma_2/2$, thus that $$|h(z)|\geq \exp(\gamma_2\sqrt{(k+|a|)\delta}|\sigma|)$$.% holds uniformly.

\noindent
{\em Case 3.} Let $z\in \big\{-k+\delta e^{\i\omega}:\,\omega\in[0,\pi]\big\}$ or  $z\in \big\{k-\delta e^{\i\omega}:\,\omega\in[0,\pi]\big\}$ (half circles outside green/yellow domains in Figure \ref{fig:esti_h}).\\
Let $z=-k+\delta e^{\i\omega}$ or $z=k-\delta e^{\i\omega}$ for $\omega=[0,\pi]$. Still let $z=a+\i b$ then $|a|=k-\delta\cos\omega$. From direct calculation,
\[
 \sqrt{k^2-z^2}=\sqrt{2k\delta e^{\i\omega}-\delta^2 e^{2\i\omega}}=r_z e^{\i \theta}
\]
where
\[
 r_z\geq \sqrt{2k\delta-\delta^2}>\frac{\sqrt{\delta(k+|a|)}}{2}\quad\text{ and }\quad \theta\in \left[0,\frac{\pi}{2}\right].
\]  
With Assumption \ref{asp2} again, there is a $\gamma_3>0$ such that $\sin(\theta+\tau)\geq \gamma_3$. Thus 
\[
 |h(z)|\geq \exp\left(\sqrt{\delta (k+|a|)}|\sigma|\gamma_3\right).
\]

\noindent
{\em Case 4.} Let  $z\in \big\{-k+\xi e^{\i\omega}:\,\omega\in[0,\pi]\big\}$ or  $z\in \big\{k-\xi e^{\i\omega}:\,\omega\in[0,\pi]\big\}$ where $\xi\in[0,\delta]$ (blue half disks domains in Figure \ref{fig:esti_h}).\\
We still let $\sqrt{k^2-z^2}:=r_z e^{\i\theta}$ then $r_z\approx \sqrt{2k\xi}$ is small since $\xi\in[0,\delta]$. Similar to Case 3, we have the following estimate:
\[
\left|\frac{2\sqrt{k^2-z^2}}{h(z)-1}\right|=\frac{2 r_z }{|h(z)|-1}\leq  \frac{2 r_z}{\exp(2\gamma_3 |\sigma|r_z)-1}.
 \]
From the mean value theorem, there is a $\epsilon\in[0,2\gamma_3|\sigma|r_z]$  such that
\[
 \exp(2\gamma_3 |\sigma|r_z)-1=2\gamma_3|\sigma|r_z\exp(\epsilon)\geq 2\gamma_3|\sigma|r_z.
\]
This implies that
\[
 \left|\frac{2\sqrt{k^2-z^2}}{h(z)-1}\right|\leq \frac{2 r_z}{2\gamma_3|\sigma|r_z}=\frac{1}{\gamma_3|\sigma|}.
\]
The above inequality holds uniformly for all $r\in[0,\delta]$ and $\omega\in[0,\pi]$, where $\delta>0$ is sufficiently small.

\vspace{0.3cm}

\noindent
We conclude our proof as follows.\\
\vspace{0.1cm}

\noindent
For 1), from the above arguments, let $\gamma:=\min\{\gamma_1,\,\gamma_2,\,\gamma_3\}$, then the following inequality holds uniformly:
\[
 |h(z)|\geq \exp\left(\gamma \sqrt{\delta}|\sigma|\sqrt{|\Re(z)|+k}\right),
\]
where $z$ lies in the area in any of the three cases. Since the expanded curve $\mathfrak{C}_{ext}$ is included in the union of the Case 1,2 and 3, we finally get the exponential decay of $|h(z)|$ with for all $z\in \mathfrak{C}_{ext}$. \\
\vspace{0.1cm}

\noindent
For 2), we only need to combine the results in Case 1,2 and 4. With the fact that $\cup_{j\in\Z}\,\overline{\B_+(-k+j,\delta)}$ and  $ \cup_{j\in\Z}\,\overline{\B_-(k+j,\delta)}$ are subsets of the union of domains in Case 1,2 and 4, the proof is finished.

\end{proof}

With this result, we are prepared to estimate the convergence of $D^{PML}(\alpha)$ to $D(\alpha)$ with respect to the parameters $\delta$ and $\sigma$.

\begin{theorem}
 \label{th:conv_D}
 The operator $D_\sigma^{PML}(\alpha)$ converges to $D(\alpha)$ uniformly with respect to $\alpha$, and satisfies the following estimation:
 \[
  \left\|D^{PML}_\sigma(\alpha)-D(\alpha)\right\|\leq C e^{-\gamma \sqrt{k\delta}|\sigma|}\text{ for all }\alpha\in \mathfrak{C},
 \]
 and 
  \[
  \left\|D^{PML}_\sigma(\alpha)-D(\alpha)\right\|\leq C |\sigma|^{-1}\text{ for all }\alpha\in \overline{\B_+(-\kappa,\delta)}\cup\overline{\B_-(\kappa,\delta)},
 \]
 where $C$ and $\gamma$ do not depend on $\alpha$ and the parameters $\delta$ and $\sigma$. Moreover, the solution $w^{PML}_\sigma(\alpha,\cdot)$ also converges to $w(\alpha,\cdot)$ uniformly:
 \begin{equation}\label{eq:conv_w}
 \left\|w^{PML}_\sigma(\alpha,\cdot)-w(\alpha,\cdot)\right\|_{\widetilde{H}^1_\p(\Omega_H^0)}\leq C e^{-\gamma\sqrt{k\delta}|\sigma|}\text{ for all }\alpha\in\mathfrak{C}
\end{equation}
and 
 \begin{equation}\label{eq:conv_w_1}
 \left\|w^{PML}_\sigma(\alpha,\cdot)-w(\alpha,\cdot)\right\|_{\widetilde{H}^1_\p(\Omega_H^0)}\leq C |\sigma|^{-1}\text{ for all }\alpha\in\overline{\B_+(-\kappa,\delta)}\cup\overline{\B_-(\kappa,\delta)}.
\end{equation}
\end{theorem}

\begin{proof}We first prove the uniform convergence \eqref{eq:conv_w}. 
 From direct calculation, for any $\phi,\,\psi\in\widetilde{H}^1_\p(\Omega_H^0)$, 
 \begin{equation*}
 \begin{aligned}
  \left<\left(D^{PML}_\sigma(\alpha)-D(\alpha)\right)\phi,\psi\right>&=-2\pi\i\sum_{j\in\Z}\sqrt{k^2-(\alpha+j)^2}\left[\coth\left(-\i\sqrt{k^2-(\alpha+j)^2}\sigma-1\right)\right]\widehat{\phi}(j)\overline{\widehat{\psi}(j)}\\
  &=-4\pi\i\sum_{j\in\Z}\frac{\sqrt{k^2-(\alpha+j)^2}}{h(\alpha+j)-1}\widehat{\phi}(j)\overline{\widehat{\psi}(j)}.
  \end{aligned}
 \end{equation*}
 Since $\phi,\,\psi\in\widetilde{H}^1_\p(\Omega_H^0)$, $\phi\big|_{\Gamma_H^0},\,\phi\big|_{\Gamma_H^0}\in H^{1/2}_\p(\Gamma_H^0)$. Thus 
 \[
  \|\phi\|^2_{H^{1/2}_\p(\Gamma_H^0)}=\sum_{j\in\Z}(1+j^2)^{1/2}\left|\widehat{\phi}(j)\right|^2<\infty,\quad \|\psi\|^2_{H^{1/2}_\p(\Gamma_H^0)}=\sum_{j\in\Z}(1+j^2)^{1/2}\left|\widehat{\psi}(j)\right|^2<\infty.
 \]
 We check the finite series with positive integer $N$:
 \[
  S_N:=-4\pi\i\sum_{j=-N}^N\frac{\sqrt{k^2-(\alpha+j)^2}}{h(\alpha+j)-1}\widehat{\phi}(j)\overline{\widehat{\psi}(j)}.
 \]
 With the result of Lemma \ref{lm:esti_h}, for all $\alpha\in \mathfrak{C}$ and $j\in\Z$, $|h(\alpha+j)|\geq \exp\left(\gamma\sqrt{\delta}|\sigma|\sqrt{k}\right)$ holds uniformly. When the parameters $|\sigma|$ is sufficiently large, we conclude that
 \[
  \left|\frac{4\pi \i}{h(\alpha+j)-1}\right|\leq e^{-\gamma\sqrt{k\delta}|\sigma|}
 \]
holds uniformly. Note that the constant $\gamma$ is adjusted. Then from Cauchy-Schwarz inequality,
 \begin{equation*}
 \begin{aligned}
 \left|S_N\right|&\leq e^{-\gamma\sqrt{k\delta}|\sigma|}\sum_{j=-N}^N \left|\sqrt{k^2-(\alpha+j)^2}\right|\left|\widehat{\phi}(j)\right|\left|\widehat{\psi}(j)\right|\\
 &\leq e^{-\gamma\sqrt{k\delta}|\sigma|}\left[\sum_{j=-N}^N\left|\sqrt{k^2-(\alpha+j)^2}\right| \left|\widehat{\phi}(j)\right|^2\right]^{1/2}\left[\sum_{j=-N}^N\left|\sqrt{k^2-(\alpha+j)^2}\right| \left|\widehat{\psi}(j)\right|^2\right]^{1/2}\\
 &\leq Ce^{-\gamma\sqrt{k\delta}|\sigma|}\left[\sum_{j=-N}^N(1+j^2)^{1/2} \left|\widehat{\phi}(j)\right|^2\right]^{1/2}\left[\sum_{j=-N}^N(1+j^2)^{1/2} \left|\widehat{\psi}(j)\right|^2\right]^{1/2}\\
 &= Ce^{-\gamma\sqrt{k\delta}|\sigma|} \|\phi\|_{H^{1/2}_\p(\Gamma_H^0)} \|\psi\|_{H^{1/2}_\p(\Gamma_H^0)},
  \end{aligned}
 \end{equation*}
 where the constant $C$ is chosen such that the inequality holds uniformly for all $\alpha\in \mathfrak{C}$ and $j\in \Z$. From trace theorem, 
\[
  \left|S_N\right|\leq Ce^{-\gamma\sqrt{k\delta}|\sigma|} \|\phi\|_{\widetilde{H}^1_\p(\Omega_H^0)} \|\psi\|_{\widetilde{H}^{1}_\p(\Omega_H^0)}.
\]
Since the above inequality holds uniformly for all positive integer $N$, let $N\rightarrow\infty$ we have:
\[
 \left|\left<\left(D^{PML}_\sigma(\alpha)-D(\alpha)\right)\phi,\psi\right>\right|\leq  Ce^{-\gamma\sqrt{k\delta}|\sigma|} \|\phi\|_{\widetilde{H}^1_\p(\Omega_H^0)} \|\psi\|_{\widetilde{H}^{1}_\p(\Omega_H^0)}.
\]
This implies that $D^{PML}_\sigma(\alpha)$ converges to $D(\alpha)$ uniformly with respect to $\alpha\in\mathfrak{C}$ and the convergence is exponential with respect to $|\sigma|$. This also implies that when $|\sigma|$ is sufficiently large, $D^{PML}_\sigma(\alpha)$ is invertible and 
\[
 \left\|w^{PML}_\sigma(\alpha,\cdot)-w(\alpha,\cdot)\right\|_{\widetilde{H}^1_\p(\Omega_H^0)}\leq C e^{-\gamma\sqrt{k\delta}|\sigma|}
\]
holds uniformly for $\alpha\in \mathfrak{C}$.
\\
\vspace{0.2cm}

Now let's focus on the proof of \eqref{eq:conv_w_1}. 
The uniform convergence with respect to $\alpha\in\overline{\B_+(-\kappa,\delta)}\cup\overline{\B_-(\kappa,\delta)}$ is proved in the similar way, with the second result in Lemma \ref{lm:esti_h}. Thus we omit it here.

\end{proof}

Similar to the modification of the inverse Bloch transform \eqref{eq:new_inverse}, we also need to modify the integral contour for the inverse transform of the function $w^{PML}_\sigma(\alpha,\cdot)$. Thus we need the following study for the dependence of $w^{PML}_\sigma(\alpha,\cdot)$ on $\alpha$.

\begin{lemma}
 \label{lm:ext_ana_pml}
For sufficiently large $|\sigma|$, the solution $w^{PML}_\sigma(\alpha,\cdot)$ is analytic with respect to $\alpha$ in small neighbourhoods of the half disks $\overline{\B_+(-\kappa,\delta)}$ and $\overline{\B_-(\kappa,\delta)}$.
\end{lemma}

\begin{proof}
We only need to consider the half disk $\overline{B_+(-\kappa,\delta)}$. 
From the proof of Theorem \ref{th:ana_ext}, $D^{-1}(\alpha)$ exists in $\B(-\kappa,\delta)\setminus Z_-$. 
From \eqref{eq:conv_w_1}, for sufficiently large $|\sigma|$ and all $\alpha\in \overline{\B_+(-\kappa,\delta)}$, $D^{PML}_\sigma(\alpha)$ is a small perturbation of $D(\alpha)$ thus is also invertible.

 From the definition, $D^{PML}_\sigma(\alpha)$ depends analytically on $\alpha\in\C$. From analytic Fredholm theory and the perturbation theory, $\left[D^{PML}_\sigma(\alpha)\right]^{-1}$ exists and depends analytically on $\alpha$ in a small neighourhood of $\overline{\B_+(-\kappa,\delta)}$.
  
  The proof for $\overline{\B_-(\kappa,\delta)}$ is similar thus is omitted.
\end{proof}

From Lemma \ref{lm:ext_ana_pml} and Cauchy integral formula, we get a similar formula as \eqref{eq:new_inverse}:
\begin{equation}
 \label{eq:pml_inverse}
 u^{PML}_\sigma(x)=\int_{-1/2}^{1/2} e^{\i\alpha x_1} w^{PML}_\sigma(\alpha,x)\d\alpha=\int_\mathfrak{C} e^{\i\alpha x_1} w^{PML}_\sigma(\alpha,x)\d\alpha,\quad x\in\Omega_H.
\end{equation}
With the above convergence analysis, we immediately obtain the convergence of 
$u^{PML}_\sigma$ to the exact solution $u$ defined by \eqref{eq:new_inverse} (equivalent to \eqref{eq:inverse}). This result is concluded in the next theorem.

\begin{theorem}
 \label{th:conv_u}
 
Suppose the wavenumber $k$ satisfies Assumption \ref{asp1} and $\kappa\in(-1/2,1/2)\setminus\{0\}$ is the rounding error. Let $\mathfrak{C}$ be the contour defined in Theorem \ref{th:modif_cont}. Let $w^{PML}_\sigma(\alpha,\cdot)$ be the solution of \eqref{eq:var_per_pml} for $\alpha\in \mathfrak{C}$ and $u^{PML}_\sigma$ is defined by \eqref{eq:pml_inverse}. Then $u^{PML}_\sigma\in \widetilde{H}^1_{loc}(\Omega_H)$ and satisfies
\[
 \left\|u^{PML}_\sigma-u\right\|_{\widetilde{H}^1(D)}\leq C \exp\left({2\pi \delta\max_{x\in D}{|x_1|}}\right)e^{-\gamma\sqrt{k\delta}|\sigma|}
\]
for any bounded subset $D$  in $\Omega_H$. 
\end{theorem}

\begin{proof}
Recall that from \eqref{eq:new_inverse}, for any $x\in D$,
\[
 u(x)=\int_\mathfrak{C} e^{\i\alpha x_1}w(\alpha,x)\d\alpha.
\]
From the choice of $\mathfrak{C}$, $\Im(\alpha)\in[-\delta,\delta]$. Thus with \eqref{eq:conv_w}, we have the following estimation:
\[
 \left\|e^{\i\alpha (\cdot)_1}\left(w^{PML}_\sigma-w\right)(\alpha,\cdot)\right\|_{H^1(D)}\leq C \exp\left({2\pi \delta\max_{x\in D}{|x_1|}}\right)e^{-\gamma\sqrt{k\delta}|\sigma|}.
\]
Then the estimation for $u^{PML}_\sigma$ is obtained directly.

\end{proof}

From above arguments, it is clear that the convergence of the solution approximated by \eqref{eq:pml_inverse}  in a bounded domain is exponential. The convergence rate is given by the parameter $|\sigma|\approx (m+1)^{-1}\lambda\rho$ where $\lambda>0$ is the thickness of the PML layer and $\rho>>1$ is the coefficient to define the polynomial $\widehat{s}$. For a numerical implementation, we need to solve $\alpha$-quasi-periodic problems \eqref{eq:pml_or} for all $\alpha\in \mathfrak{C}$ and then approximate the contour integral \eqref{eq:pml_inverse}.

% 
% 
% 
% From Lemma \ref{lm:ext_ana_pml}, with Cauchy integral theorem, 
% \begin{equation}\label{eq:inverse_pml}
%  u^{PML}_\sigma(x)=\int_\mathfrak{C} e^{\i\alpha x_1}w^{PML}_\sigma(\alpha,x)\d\alpha=\int_{-1/2}^{1/2} e^{\i\alpha x_1}w^{PML}_\sigma(\alpha,x)\d\alpha.
% \end{equation}
% This implies that the technique to change the integral contour is actually not necessary in numerical computations. 

\section{Numerical results}

In this section, we present numerical examples to show the convergence of the PML method. In these numerical examples, the  periodic surface is defined by the function:
\[
 \zeta(x_1)= 1.5+\frac{\sin(x_1)}{3}-\frac{\cos(2x_1)}{4}.
\]
The source term is also fixed:
\begin{equation*}
f(x)=\begin{cases}
      0,\quad |x-a_0|>0.3;\\
      3,\quad 0.1<|x-a_0|<0.3;\\
      3\zeta(|x-a_0|),\quad\text{otherwise.}
     \end{cases}
\end{equation*}
Here $\zeta(t)$ is a $C^8$-continuous cutoff function which equals to $1$ when $t\leq 0.1$ and $0$ when $t\geq 0.3$, and $a_0=(0,1.8)$. Note that $f$ is compactly supported in the disk with center $a_0$ and radius $0.3$. The height $H$ is taken as $2.5$ and the the thickness of the PML layer $\lambda$ is fixed as $1.5$. The fixed complex $\chi=\exp(\i\pi/4)$. For structures and the source term we refer to Figure \ref{fig:sample:eg}.

\begin{figure}[tttttt!!!b]
\centering
\begin{tabular}{cc}
\includegraphics[width=0.45\textwidth]{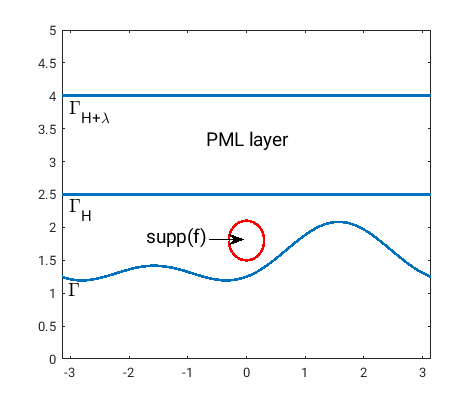} & 
\includegraphics[width=0.45\textwidth]{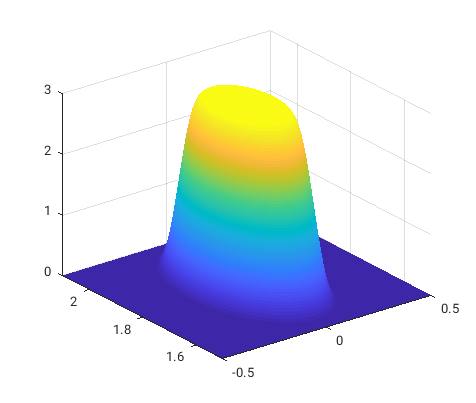} \\
(a) & (b) 
\end{tabular}
\caption{(a)  Structure; (b) source term.}
\label{fig:sample:eg}
\end{figure}

We produce the ``exact solution'' (denoted by $u_{ext}$) by the numerical approximation of the exact formulation \ref{eq:per1}-\eqref{eq:per3} with the discretization method of \eqref{eq:inverse} (with 80 nodal points) introduced in \cite{Zhang2017e}. The parameter $H$ here is chosen as $4$, which is different from the PML method which is $2.5$. The maximum meshsize is $0.005$ and the DtN map is approximated by a finite series 
\[
 T^+_{\alpha,N}\phi:=\i\sum_{j=-40}^40 \sqrt{k^2-(\alpha+j)^2}\widehat{\phi}_j e^{\i j x_1}\text{ where }\phi(x_1)=\sum_{j\in\Z}\widehat{\phi}_j e^{\i  j x_1}.
\]
Then we compare the numerical result obtained by the PML method with different parameter $\rho$'s with the ``exact solution'', on a straight line $S:=(-\pi,\pi)\times\{2.4\}$, which lies below the PML layer. Note that since the thickness $\lambda$ is fixed, the parameter $\sigma$ only depends on $\rho$. Thus we replace the subscription $\sigma$ by $\rho$ in this section. The relative error is defined as
\[
 err(\rho)=\frac{\left\|u^{PML}_\rho-u_{ext}\right\|_{L^2(S)}}{\left\|u_{ext}\right\|_{L^2(S)}}.
\]
Note that the meshes are exactly the same as for the ``exact solutions'' and the discretization method of \eqref{eq:inverse_pml} is introduced in \cite{Zhang2019b} with also $80$ points.

We carry out the numerical methods for four different wavenumbers. Two wavenumbers satisfy Assumption \ref{asp1}, which are $1.2$ and $\sqrt{5}$; and two do not satisfy this assumption, which are $1$ and $1.5$. Numerical results with different $\rho$'s are listed in Table \ref{eg}. We also plot the logarithm  relative error against the parameter $\rho$ in Figure \ref{fig:eg} (a). From both Table \ref{eg} and Figure \ref{fig:eg} (a), the error decays exponentially at first, and the decay no longer days when it reaches  $~10^{-5}$. We give two possible reasons for this phenomenon. The first is the error from the finite element method with fixed meshes in both the ``exact solutions'' and the PML solutions. Note that since the convergence rate for the discretization methods introduced in \cite{Zhang2017e,Zhang2019b} is always very fast, we ignore the relative errors from these processes. The second reason is the increasing of errors due larger parameter $\rho$'s. 

\begin{table}[htb]
\centering
\caption{Relative $L^2$-errors different $k$'s and $\rho$'s.}\label{eg}
\begin{tabular}
{|p{1.8cm}<{\centering}||p{2cm}<{\centering}|p{2cm}<{\centering}
 |p{2cm}<{\centering}|p{2cm}<{\centering}|}
\hline
  & $k=1.2$ & $k=\sqrt{5}$ & $k=1$ & $k=1.5$\\
\hline
\hline
$\rho=2$&$2.18$E$-01$&$4.97$E$-02$&$3.12$E$-01$&$1.47$E$-01$\\
\hline
$\rho=4$&$3.52$E$-02$&$2.04$E$-03$&$5.61$E$-02$&$1.77$E$-02$\\
\hline
$\rho=6$&$6.10$E$-03$&$8.94$e$-05$&$1.22$E$-02$&$1.56$E$-03$\\
\hline
$\rho=8$&$1.03$E$-03$&$2.75$e$-05$&$2.77$E$-03$&$2.06$E$-04$\\
\hline
$\rho=10$&$1.71$E$-04$&$3.12$e$-05$&$6.43$E$-04$&$4.98$E$-05$\\
\hline
$\rho=12$&$3.15$E$-05$&$3.31$e$-05$&$1.48$E$-04$&$3.01$E$-05$\\
\hline
$\rho=14$&$2.21$E$-05$&$3.50$e$-05$&$3.27$E$-05$&$2.78$E$-05$\\
\hline
$\rho=16$&$2.20$E$-05$&$3.72$e$-05$&$1.92$E$-05$&$2.86$E$-05$\\
\hline
$\rho=18$&$2.53$E$-05$&$3.96$e$-05$&$1.74$E$-05$&$2.96$E$-05$\\
\hline
$\rho=20$&$2.86$E$-05$&$4.21$e$-05$&$1.85$E$-05$&$3.13$E$-05$\\
\hline
\end{tabular}
\end{table}

\begin{figure}[tttttt!!!b]
\centering
\begin{tabular}{cc}
\includegraphics[width=0.45\textwidth]{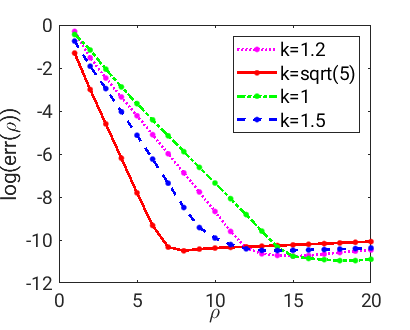} & 
\includegraphics[width=0.45\textwidth]{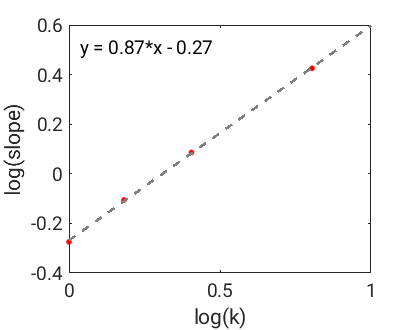} \\
(a) & (b) 
\end{tabular}
\caption{(a) dependence of errors on the parameter $\rho$; (b) slopes with difference $k$'s.}
\label{fig:err}
\end{figure}

It is interesting to see that even for wavenumbers which do not satisfy Assumption \ref{asp1}, the PML method also converges exponentially with respect to the parameter $\rho$. This may imply that the error estimate is expected to be extended to these cases. We also observe an increasing of the slopes with larger $k$'s. We carry out line fittings for each curve in the exponentially decaying parts and the results are shown in Table \ref{slope} and Figure \ref{fig:err} (b). In Theorem \ref{th:conv_u} it is expected that the dependence of the slope is $\sqrt{k}$, the result shown in Figure \ref{fig:err} is even faster. This may come from the difference choices of another parameter $\delta$, which is not very clear according to the discussions in this paper. 

\begin{table}[htb]
\centering
\caption{Slopes with different $k$'s.}\label{slope}
\begin{tabular}
{|p{1.8cm}<{\centering}||p{1.2cm}<{\centering}|p{1.2cm}<{\centering}
 |p{1.2cm}<{\centering}|p{1.2cm}<{\centering}|}
\hline
wavenumber  & $k=1$& $k=1.2$   & $k=1.5$&  $k=\sqrt{5}$\\
\hline
\hline
slopes & $0.76$ & $0.90$ & $1.09$   & $1.53$\\
\hline
\end{tabular}
\end{table}

\section{Further comments}

%\begin{remark}
 The method introduced in this paper can be extended without major difficulty to the case with local perturbations in the periodic surface. However, we do not discuss this case in this paper since we would like to have simplified representations.  For details we refer to \cite{Zhang2017e}. This method can  also be extended to locally perturbed periodic layers, but this may involve the guided modes which propagate along the periodic structures. We refer to \cite{Fliss2021} for some details for this case.

 From numerical examples, for wavenumbers that do not satisfy Assumption \ref{asp1} the convergence is also exponential. The decay rates for all the wavenumbers are much faster than expected (since from Theorem \ref{th:conv_u} the convergence depends on $\delta>0$ which is expected to be very small), which implies that the estimation in this paper maybe not optimal. Due to above reasons, the author has a conjecture that the convergence rate does not depend on $\delta$ thus it is easily extended to the case with Assumption \ref{asp1}.  Since we are not able to prove that at present, it remains to be an open question.
 
 %In this paper, we only focus on the Dirichlet boundary condition as an example. Other boundary conditions can also be considered in the same way, such as Neumann boundary condition or (locally perturbed) periodic  impedance boundary condition.
%\end{remark}
 \section*{Acknowledgment}
 This work is funded by the Deutsche Forschungsgemeinschaft (DFG, German Research Foundation) – Project-ID 258734477 – SFB 1173 and Project-ID 433126998. The idea to apply the perfectly matched layers to the periodic open waveguide problems was proposed by Prof. Sonia Fliss during a discussion. The author is grateful for her valuable suggestions that motivated this paper. The author is also grateful for Prof. Andreas Kirsch for his valuable suggestions in improving this paper.

\bibliographystyle{plain}
\bibliography{ip-biblio}
\end{document}